\documentclass{amsart}
\usepackage[american]{babel}
\usepackage[T1]{fontenc}
\usepackage[utf8]{inputenc}
\usepackage{amsmath,amsfonts,amstext, dsfont}
\usepackage{amsthm,amssymb,amsmath,amscd,graphicx,epsfig,hhline,mathrsfs,latexsym,url}
\usepackage[neverdecrease]{paralist}
\usepackage{yfonts,dsfont}
\usepackage{tikz}
\usetikzlibrary{shapes.geometric, arrows}
\newcommand{\mc}{\mathcal}
\newcommand{\ms}{\mathscr}

\newcommand{\mb}{\mathbb}
\newcommand{\mr}{\mathrm}

\usepackage{amsmath, amsthm, amscd, amsfonts, amssymb, graphicx, color}
\usepackage[bookmarksnumbered, colorlinks, plainpages]{hyperref}
\hypersetup{colorlinks=true,linkcolor=red, anchorcolor=green, citecolor=cyan, urlcolor=red, filecolor=magenta, pdftoolbar=true}

\newtheorem{Thm}{Theorem}[section]

\newtheorem{Le}[Thm]{Lemma}

\theoremstyle{definition}

\theoremstyle{remark}
\newtheorem*{Rem}{Remark}

\title[{Counter-examples to the D-S pointwise ergodic theorem on $L^1+L^\infty$}]{Counter-examples to the Dunford-Schwartz pointwise ergodic theorem on $L^1+L^\infty$}

\author{D\'avid Kunszenti-Kov\'acs}
\address{MTA Alfréd Rényi Institute of Mathematics, P.O. Box 127, H-1364 Budapest, Hungary}
\email{\textcolor[rgb]{0.00,0.00,0.84}{daku@renyi.hu}}


\begin{document}

\maketitle

\begin{abstract} Extending a result by Chilin and Litvinov, we show by construction that given any $\sigma$-finite infinite measure space $(\Omega,\mc{A}, \mu)$ and a function $f\in L^1(\Omega)+L^\infty(\Omega)$ with $\mu(\{|f|>\varepsilon\})=\infty$ for some $\varepsilon>0$, there exists a Dunford-Schwartz operator $T$ over $(\Omega,\mc{A}, \mu)$ such that $\frac{1}{N}\sum_{n=1}^N (T^nf)(x)$ fails to converge for almost every $x\in\Omega$. In addition, for each operator we construct, the set of functions for which pointwise convergence fails almost everywhere is residual in $L^1(\Omega)+L^\infty(\Omega)$.

\end{abstract}


\section{Introduction}

For sigma-finite measure spaces $X=(\Omega,\mc{A},\mu)$ that have infinite total measure, operator theoretic aspects of ergodic theory become much more complicated than for probability spaces (or, equivalently, finite measure spaces). On the one hand, the Dunford-Schwartz pointwise ergodic theorem holds, i.e., for any function $f\in L^1(X)$, and any Dunford-Schwartz operator over $\Omega$, the ergodic averages $\frac{1}{N}\sum_{n=1}^N T^nf$ converge almost everywhere (cf. \cite[Theorem VIII.6.6.]{DS}). On the other hand, the same averages need not converge in norm, i.e., the mean ergodic theorem fails in general. In the background lies the non-equivalence of mean ergodicity and weak almost periodicity (orbits of funtions being weakly sequentially relatively compact), boiling down to $\mathds{1}_{\Omega}\notin L^1(X)$ (cf. \cite{KL}). In particular, there is no Jacobs-Glicksberg-deLeeuw-type decomposition (cf., e.g., \cite[Theorem 1.15]{eisner-book}) for general DS operators on $L^1(X)$.

In a recent paper (\cite{CL}), V. Chilin and S. Litvinov investigated pointwise ergodic theorems on infinite measure spaces for functions beyond the classical $L^1$ space. Namely, they determined which functions in the symmetric space $L^1+L^\infty$ yield pointwise almost everywhere convergence for all Dunfod-Schwartz operators under the assumption that the infinite sigma-finite measure space is quasi-non-atomic, i.e., it only contains atoms that have the same measure. In a companion paper (\cite{CCL}) with a third author,  D. \c{C}ömez, it was proved that the Dunford-Schwartz pointwise ergodic theorem holds on symmetric spaces $E$ contained in $L^1+L^\infty$ provided the constant $1$ function is not in $E$, for any sigma-finite measure space. For quasi-non- atomic measures these two papers give a complete characterization of functions for which the Dunford-Schwartz theorem holds true, but the general sigma-finite case was left open. Our aim in this paper is to close this gap, and show that the same result holds without restriction on the atomic part of the measure.

The open question pertains to functions that do not ``decay'' in the sense that for some interval $[a,b]\subset (0,\infty)$, the level set $\left\{|f|\in[a,b]\right\}$ has infinite measure. The original approach used in \cite{CL} was to first provide a counter-example separately for the Lebesgue measure on the positive half-line and for the ``exceptional case'' $\ell^\infty(\mb{N})$. Then the general result followed by decomposing the original measure into its atomic and non-atomic parts, and using isomorphisms between complete Boolean algebras to reduce to the already proven two special cases. Key to their construction  was the idea of using measure-preserving point maps and the corresponding Koopman operators, perturbed by a suitably chosen multiplication operator.\\
Our way of approaching the structure of the measure space is in some sense the opposite: we do not aim at transforming the non-atomic part into its standard Lebesgue space form and then add on a uniform atomic part when necessary, but rather consider fully atomic measures as our base case, and reduce general sigma-finite measure spaces to fully atomic ones by factorisation, stepping away from measure preserving maps and corresponding Koopman operators on the original space.\\
We also show that the counter-example operator that we construct for a specific function is actually a counter-example in the strongest sense possible in that it yields Cesàro averages that fail to converge not only on a set of positive measure, but almost everywhere. In addition, we show that each constructed operator actually has this same property for any ``typical'' function in $L^1+L^\infty$ -- in the Baire category sense.

\section{Results}

\begin{Thm}\label{thm:f}
Let $X=(\Omega, \mc{A},\mu)$ be a $\sigma$-finite infinite measure space, and $f\in L^1(X)+L^\infty(X)$. Suppose that there exists an $\varepsilon>0$ such that $\mu(\{|f|>\varepsilon\})=\infty$. Then there exists a Dunford-Schwartz operator $T$ over $X$ such that $\frac{1}{N}\sum_{n=1}^N (T^nf)(x)$ fails to converge for almost every $x\in\Omega$.
\end{Thm}
\begin{proof}
By definition there exist functions $f_1\in L^1(X)$ and $f_2\in L^\infty(X)$ with $f=f_1+f_2$, and so we have that $\mu(\{|f|>2\|f_2\|_\infty\})<\infty$. We may assume the existence of an $\varepsilon\in(0,2\|f_2\|_\infty)$ with $\mu(\{|f|>\varepsilon\})=\infty$. This then implies that there exists a $z_0\in\mb{C}\backslash \{0\}$ such that $\mu(\{\Re (f/z_0) \in[1/2,1]\})=\infty$, since finitely many such sets cover $\{|f|\in[\varepsilon,2\|f_2\|_\infty]\}$. Let us write $A:=\{\Re (f/z_0) \in[1/2,1]\}$. By the $\sigma$-finiteness of $X$ together with $\mu(A)=\infty$, there exists a countable collection of pairwise disjoint sets of positive measure $H_{j,n}$ with $j\in\mc{J}$ and $n\in\mb{N}$ such that
\begin{align*}
\bigcup_{j\in\mc{J},n\in\mb{N}} H_{j,n}= \Omega, &&\bigcup_{j\in\mc{J}, n\in\mb{N}^+} H_{j,n}\subset A,
\end{align*}
and for each $j\in\mc{J}$ and $n\in\mb{N}$ we have
\begin{align*}
\mu(H_{j,n+1})\in[\mu(H_{j,n}), \infty).
\end{align*}

Now define the operator $T:L^1(X)+L^\infty(X)\to L^1(X)+L^\infty(X)$ as follows.
\begin{align*}
(Tg)(x):= \frac{(-1)^{\mathds{1}_{\log_3 (n+1)\in\mb{N}}}}{\mu(H_{j,n+1})}\int_{H_{j,n+1}} g d\mu && \mbox{ when } x\in H_{j,n}.
\end{align*}

It follows easily from the way $T$ is defined that $T|_{L^p}:L^p\to L^p$ is a contraction for every $1\leq p\leq \infty$, hence it is a Dunford-Schwartz operator over $X$. It remains to be shown that
\[
\frac{1}{N}\sum_{m=1}^N (T^mf)(x)
\]
fails to converge for almost every $x\in\Omega$. 
It is sufficient to instead show that
\[
\frac{1}{N}\sum_{m=1}^N \Re\left((T^mf)(x)/z_0\right)
\]
fails to converge for almost all $x\in\Omega$.

To this end, note that for any $m\in\mb{N}^+$, $j\in\mc{J}$, $n\in\mb{N}$ and $x\in H_{j,n}$ we have
\begin{align*}
(T^mf)(x)&= \frac{\prod_{k=1}^m(-1)^{\mathds{1}_{\log_3 (n+k)\in\mb{N}}}}{\mu(H_{j,n+m})}\int_{H_{j,n+m}} f d\mu\\
&= \frac{(-1)^{\lfloor \log_3 (n+m) \rfloor-\lfloor \log_3 n \rfloor}}{\mu(H_{j,n+m})}\int_{H_{j,n+m}} f d\mu.
\end{align*}
Let $b:=\lfloor \log_3 n \rfloor$. Since $\Re (f/z_0) \in[1/2,1]$ on $H_{j,k}$ for all $k\geq 1$, we then have for any $\ell\in\mb{N}^+$ that on the one hand
\begin{align*}
&\frac{1}{3^{b+2\ell}-n-1}\sum_{m=1}^{3^{b+2\ell}-n-1} \Re\left((T^mf)(x)/z_0\right)\\
=&\frac{1}{3^{b+2\ell}-n-1}\sum_{m=n+1}^{3^{b+2\ell}-1} \frac{(-1)^{\lfloor \log_3 m \rfloor-\lfloor \log_3 n \rfloor}}{\mu(H_{j,m})}\int_{H_{j,m}} \Re\left(f/z_0\right) d\mu\\
\leq&
\frac{1}{3^{b+2\ell}-n-1}
\sum_{d=0}^{2\ell-1}\,
\sum_{m=0}^{2\cdot3^{b+d}-1} \frac{(-1)^d}{\mu(H_{j,3^{b+d}+m})}\int_{H_{j,3^{b+d}+m}} \Re\left(f/z_0\right) d\mu\\
\leq&
\frac{1}{3^{b+2\ell}-n-1}
\sum_{a=0}^{\ell-1}\left(
2\cdot3^{b+2a}\cdot 1 - 2\cdot3^{b+2a+1}\cdot \frac{1}{2}\right)\\
\leq&
\frac{-3^{b+2(\ell-1)}}{3^{b+2\ell}-n-1}\leq \frac{-3^{b+2(\ell-1)}}{3^{b+2\ell}}=-1/9,
\end{align*}
and on the other hand
\begin{align*}
&\frac{1}{3^{b+2\ell+1}-n-1}\sum_{m=1}^{3^{b+2\ell+1}-n-1} \Re\left((T^mf)(x)/z_0\right)\\
=&\frac{1}{3^{b+2\ell+1}-n-1}\sum_{m=n+1}^{3^{b+2\ell+1}-1} \frac{(-1)^{\lfloor \log_3 m \rfloor-\lfloor \log_3 n \rfloor}}{\mu(H_{j,m})}\int_{H_{j,m}} \Re\left(f/z_0\right) d\mu\\
\geq&
\frac{1}{3^{b+2\ell+1}-n-1}
\sum_{d=1}^{2\ell}\,
\sum_{m=0}^{2\cdot3^{b+d}-1} \frac{(-1)^d}{\mu(H_{j,3^{b+d}+m})}\int_{H_{j,3^{b+d}+m}} \Re\left(f/z_0\right) d\mu\\
\geq&
\frac{1}{3^{b+2\ell+1}-n-1}
\sum_{a=1}^{\ell}\left(
2\cdot3^{b+2a}\cdot \frac{1}{2} - 2\cdot3^{b+2a-1}\cdot 1\right)\\
\geq&
\frac{3^{b+2\ell-1}}{3^{b+2\ell+1}-n-1}\geq \frac{3^{b+2\ell-1}}{3^{b+2\ell+1}}=1/9.
\end{align*}

In other words, for almost all $x\in\Omega$, we have
\[
\limsup\frac{1}{N}\sum_{m=1}^N \Re\left((T^mf)(x)/z_0\right)\geq 1/9
\]
and
\[
\liminf\frac{1}{N}\sum_{m=1}^N \Re\left((T^mf)(x)/z_0\right)\leq -1/9,
\]
hence pointwise convergence of $\frac{1}{N}\sum_{m=1}^N (T^mf)$ fails in almost every point of $\Omega$.

\end{proof}

Next, we shall have a closer look at the structure of this operator $T$.

Let $O:=\mc{J}\times\mb{N}$, and $\varphi:\Omega\to O$ be the factor map $x\mapsto (j,n)$ whenever $x\in H_{j,n}$ ($j\in\mc{J}, n\in\mb{N}$). This map induces a fully atomic push-forward measure $\nu$ on $O$ with $\nu(\{(j,n)\}):=\mu(H_{j,n})$. Further, with the notation $Y:=(O,\ms{P}(O),\nu)$, on the level of functions, we have the natural operators 
\[
P:L^1(X)+L^\infty(X)\to L^1(Y)+L^\infty(Y)
\]
and
\[
Q:L^1(Y)+L^\infty(Y)\to L^1(X)+L^\infty(X)
\]
defined through
\begin{align*}
(Pg)(j,n)=&\int_{H_{j,n}} g\,d\mu, \mbox{ and }\\
(Qv)(x)={}& v(j,n) \,\forall x\in H_{j,n}
\end{align*}
for all $j\in\mc{J}, n\in\mb{N}$. This allows us to define the operator $S:=PTQ$.

Consider the Koopman operator $K$ induced by the the left shift acting on the $\mb{N}$ component of $O$, and the multiplication operator $M_\psi:L^p(Y)\to L^p$ (for all $p$) which multiplies by the function $\psi\in\ell^\infty(O)$ where $\psi$ takes the value $-1$ on pairs $(j,n)$ where $n\in\mb{N}$ is a power of $3$ and the value $1$ at all other points. Then from the definition of these operators we may see that also $S=KM_\psi$ holds.

\begin{Le}\label{le:non_conv}
For any function $g\in L^1(X)+L^\infty(X)$, the averages $\frac{1}{N}\sum_{n=1}^N T^ng$ converge almost nowhere if and only if the averages $\frac{1}{N}\sum_{n=1}^N S^n(Pg)$ converge nowhere.
\end{Le}
\begin{proof}
To this end, note that $P$ is an isometric isomorphism from the range of $T$ on $L^1(X)+L^\infty(X)$ to $L^1(Y)+L^\infty(Y)$, and its inverse is $Q$, with both $Q$ and $P$ preserving pointwise behaviour on the spaces in question. Hence $T=QSP$, and $\frac{1}{N}\sum_{n=1}^N T^ng$ converges almost everywhere if and only if so does
\[
\frac{1}{N}\sum_{n=1}^N T^ng=\frac{1}{N}\sum_{n=1}^N (QSP)^ng=Q\frac{1}{N}\sum_{n=1}^N S^n(Pg).
\]
\end{proof}

Given a function $v\in L^1(Y)+L^\infty(Y)$, let $\mc{C}(v;j,n)\subset\mb{C}$ denote the set of accumulation points of the sequence $\left(\frac{1}{N}\sum_{k=1}^N (S^kv)(j,n)\right)_{N\in\mb{N}^+}$, and let
\[
d(v;j,n):=\mr{diam} (\mc{C}(v;j,n)).
\]
We then have the following result.

\begin{Le}\label{le:diam}
For any $v\in L^1(Y)+L^\infty(Y)$ and $j\in\mc{J}$, we have that $\mc{C}(v;j,n)$ is independent of $n\in\mb{N}$, and the sequence $\frac{1}{N}\sum_{k=1}^N (S^kv)(j,n)$ converges if and only if $d(v;j,0)=0$.
\end{Le}
\begin{proof}
The first part follows from the fact that for nonnegative integers $n<m$ and $k\in\mb{N}^+$, we have
\begin{equation*}
(S^kv)(j,m)=(S^{k+m-n}v)(j,n)\cdot(-1)^{\left|\left\{n<\ell\leq m,\, \ell\in\mb{N}, \,\log_3\ell\in\mb{N}\right\}\right|}.
\end{equation*}
For the second part, we only have to show that the sequence $\frac{1}{N}\sum_{k=1}^N (S^kv)(j,0)$ is always bounded. Since $\nu(\{(j,\ell)\})$ is monotone increasing in $\ell$, we have that $(v(j,\ell))_{\ell\in\mb{N}}\in\ell^\infty(\mb{N})$. Also, $\|\psi\|_\infty=1$, hence 
\[
\left|\frac{1}{N}\sum_{k=1}^N (S^kv)(j,0)\right|\leq\|(v(j,\ell))_{\ell\in\mb{N}}\|_\infty
\]
\end{proof}

With these results at hand, we are ready to show that $T$ is a counter-example not only for the function $f$, but also provides almost nowhere convergence of the ergodic means for a large class of functions, in the Baire category sense.

\begin{Thm}
The set
\[
\ms{H}:=\left\{
g\in L^1(X)+L^\infty(X)\left| \frac{1}{N}\sum_{n=1}^N T^ng \mbox{ is almost nowhere convergent }
\right.
\right\}
\]
is residual in $L^1(X)+L^\infty(X)$.

\end{Thm}
\begin{proof}
Consider the closed subspace
\[
U:=\left\{
g\in L^1(X)+L^\infty(X)\left|
\int_{H_{j,n}}g\,d\mu=0 \,\forall j\in\mc{J},n\in\mb{N}
\right.
\right\}.
\]
Then it is clear that we have the direct sum decomposition
\[
L^1(X)+L^\infty(X)=U\oplus\mr{rg}(Q)
\]
with $\mr{rg}(Q)$ also closed. Also, we have $U\subset\rm{ker}(T)$, whence
\[
\ms{H}=(\ms{H}\cap\mr{rg}(Q))+U.
\]
Thus to show that $\ms{H}$ is residual, it is enough to prove that $\ms{H}\cap\mr{rg}(Q)$ is residual in $\mr{rg}(Q)$.
However, we have seen that the subspace $\mr{rg}(Q)$ is isometrically isomorphic to $L^1(Y)+L^\infty(Y)$, and by Lemma \ref{le:non_conv},
\[
P\ms{H}=\left\{
v\in L^1(Y)+L^\infty(Y)\left| \frac{1}{N}\sum_{n=1}^N T^ng \mbox{ is nowhere convergent }
\right.
\right\}=:\ms{G},
\]
so this is equivalent to showing that $\ms{G}$ is residual in $L^1(Y)+L^\infty(Y)$.

Actually, we shall instead show that a somewhat smaller set,
\[
\ms{G}_0=\left\{
v\in L^1(Y)+L^\infty(Y)\left| \inf_{j\in\mc{J}} d(v;j,0)>0
\right.
\right\},
\]
is already residual. The fact that $\ms{G}_0\subset\ms{G}$ follows from Lemma \ref{le:diam}.\\
Let $v_0\in\ms{G}_0$, and $\varepsilon:=\inf_{j\in\mc{J}} d(v_0;j,0)$, and $w\in L^1(Y)+L^\infty(Y)$ with
\[
\|w\|_{L^1(Y)+L^\infty(Y)}<\varepsilon/3.
\]
Then there exist $w_1\in L^1(Y)$ and $w_2\in L^\infty(Y)$ such that $w=w_1+w_2$ and $\|w_1\|_1,\|w_2\|_\infty<\varepsilon/3$. Since for each $j\in\mc{J}$ we have that $\nu(\{(j,n)\})$ is monotone increasing, we also have that
\[
\left(
(w_1(j,n)
\right)_{n\in\mb{N}}\in c_0(\mb{N}),
\]
meaning that for each $j\in\mc{J}$
\begin{align*}
\lim_{n\to\infty} (S^nw_1)(j,0)\cdot(-1)^{\left|\left\{0<\ell\leq n,\, \ell\in\mb{N}, \,\log_3\ell\in\mb{N}\right\}\right|}=\lim_{n\to\infty}w_1(j,n)=0.
\end{align*}
Thus for any $(j,n)\in O$ we have $\lim_{N\to\infty}\frac{1}{N}\sum_{k=1}^N (S^kw_1)(j,n)=0$. On the other hand, since $S$ is contractive on $L^\infty(Y)$, we also have $d(w_2;j,n)\leq2\varepsilon/3$ for all $(j,n)\in O$. But then we have
\begin{align*}
d(v_0+w;j,n)&=d(v_0+w_1+w_2;j,n)=d(v_0+w_2;j,n)\\
&\geq d(v_0;j,n)-d(w_2;j,n)\geq \varepsilon/3,
\end{align*}
implying $v_0+w\in\ms{G}_0$. This shows that $\ms{G}_0$ is an open set in $L^1(Y)+L^\infty(Y)$. It now only remains to be shown that its complement contains no open ball.\\
Let $v_1\in L^1(Y)+L^\infty(Y)\backslash\ms{G}_0$ and $\delta>0$ be arbitrary. Let further
\[
\mc{J}_1:=\left\{
j\in\mc{J}\left|
d(v_1;j,0)<\delta/9
\right.
\right\},
\]
which by the choice of $v_1$ is not empty. Note that by the construction of $S$ and the proof of Theorem \ref{thm:f}, the characteristic function $\mathds{1}_{\mc{J}_1\times\mb{N}}\in L^1(Y)+L^\infty(Y)$ is such that $\|\mathds{1}_{\mc{J}_1\times\mb{N}}\|_{L^1(Y)+L^\infty(Y)}=1$ and for each $j_1\in\mc{J}_1$ we have $d(\mathds{1}_{\mc{J}_1\times\mb{N}};j_1,0)\geq 2/9$, whereas $d(\mathds{1}_{\mc{J}_1\times\mb{N}};j_2,0)=0$ for all $j_2\not\in\mc{J}_1$. Consequently, we have
\begin{align*}
&\inf_{j\in \mc{J}}d\left(v_1+\delta \mathds{1}_{\mc{J}_1\times\mb{N}};j,0\right)\\
={}&\min\left\{
\inf_{j\in \mc{J}_1}d\left(v_1+\delta \mathds{1}_{\mc{J}_1\times\mb{N}};j,0\right)
;
\inf_{j\not\in \mc{J}_1}d\left(v_1+\delta \mathds{1}_{\mc{J}_1\times\mb{N}};j,0\right)
\right\}\\
={}&\min\left\{
\inf_{j\in \mc{J}_1}d\left(v_1+\delta \mathds{1}_{\mc{J}_1\times\mb{N}};j,0\right)
;
\inf_{j\not\in \mc{J}_1}d\left(v_1;j,0\right)
\right\}\\
\geq&\min\left\{
\inf_{j\in \mc{J}_1}d\left(v_1+\delta \mathds{1}_{\mc{J}_1\times\mb{N}};j,0\right)
;
\delta/9
\right\}\\
\geq&\min\left\{
\delta \cdot\frac{2}{9}
-\inf_{j\in \mc{J}_1}d(v_1;j,0)
;
\delta/9
\right\}=\delta/9.
\end{align*}
This means that the open ball of radius $2\delta$ around $v_1$ intersects $\ms{G}_0$, and we are done.

\end{proof}

\begin{Rem}
If $\mc{J}$ is finite, then the set $\ms{G}_0$ is actually all of $\ms{G}$.
\end{Rem}

\section*{Acknowledgements}

The author has received funding from the European Research Council under the European Union's Seventh Framework Programme (FP7/2007-2013) / ERC grant agreement $\mr{n}^\circ$617747, and from the MTA R\'enyi Institute Lend\"ulet Limits of Structures Research Group.

\end{document}